\documentclass[12pt,a4paper,reqno]{amsart}

\numberwithin{equation}{section}

\usepackage[top=3cm, bottom=3cm, left=3cm, right=3cm]{geometry}
\usepackage[latin1]{inputenc}
\usepackage{csquotes}
\usepackage{amsmath}
\usepackage{amsthm}
\usepackage{amsfonts}
\usepackage{amssymb}
\usepackage{graphics}
\usepackage{float}
\usepackage[hang,flushmargin]{footmisc} 

\usepackage{amsmath,amsthm,amssymb,graphicx, multicol, array}
\usepackage{enumerate}
\usepackage{enumitem}

\usepackage{xcolor}

\usepackage[pagebackref,bookmarks,colorlinks,breaklinks]{hyperref}

\hypersetup{linkcolor=blue,citecolor=red,filecolor=blue,urlcolor=blue} 

\usepackage{tabto}

\numberwithin{equation}{section}

\usepackage{tikz}
\usepackage{pgfplots}

\usetikzlibrary{matrix}

\usepackage{mathrsfs}

\DeclareMathOperator{\lcm}{lcm}

\newtheorem{thm}{Theorem}[section]
\newtheorem{theorem}{Theorem}[section]
\newtheorem{lem}{Lemma}[section]

\newcommand{\N}{\mathbb{N}}

\newcommand{\C}{\mathbb{C}}

\usepackage{fancyhdr}
\newlength{\bibitemsep}
\newlength{\bibparskip}
\let\oldthebibliography\thebibliography
\renewcommand\thebibliography[1]{%
	\oldthebibliography{#1}%
	\setlength{\parskip}{\bibitemsep}%
	\setlength{\itemsep}{\bibparskip}%
}

\begin{document}

\title{Autocorrelation of the Mobius Function}


\author{N. A. Carella}
\address{}
\curraddr{}
\email{}
\thanks{}


\subjclass[2010]{Primary 11N37; Secondary 11P32.}

\keywords{Autocorrelation, Correlation, Arithmetic function, Mobius function, Liouville function, Chowla conjecture.}

\date{}

\dedicatory{}

\begin{abstract} Let $x\geq 1$ be a large integer, and let $\mu:\mathbb{N}\longrightarrow\{-1,0,1\}$ be the Mobius function. This article proposes an effective asymptotic result for the autocorrelation function $\sum_{n \leq x} \mu(n) \mu(n+t) =O\left( e^{-c\sqrt{\log x}}\right) $, where $t\ne 0$ be a small fixed integer, and $c>0$ is a constant. 
\end{abstract}

\maketitle
\tableofcontents
\pagenumbering{Page gobble}
\pagenumbering{arabic}

\section{Introduction} \label{S5757}
The correlation function induced by a pair of arithmetic functions $f,g:\mathbb{N}\longrightarrow\mathbb{C}$ is defined by the arithmetic average
\begin{equation} \label{eq5757.200}
	\sum_{n \leq x}f(n) g(n).
\end{equation}	

The Chowla conjecture, Elliot conjecture, and Sarnak conjecture are correlation problems associated with the Mobius function $g(n)=\mu(n)$ and bounded multiplicative functions $f(n) \ll1$.  These problems, which are the topics of current research, have many similarities and common structures. The corresponding arithmetic averages are, respectively,
\begin{equation} \label{eq5757.200A}
	\sum_{n \leq x} \mu(n) \mu(n+t)=o(x), 
\end{equation}	
where $t\ne 0$, see \cite{CS1965}, \cite{RO2018};
\begin{equation} \label{eq5757.200B}
	\sum_{n \leq x} f(n) f(n+t)=o(x), 
\end{equation}	
where $f:\N\longrightarrow \C$ is a bounded arithmetic function, \cite{EP1994}, \cite{KT2023}, and 
\begin{equation} \label{eq5757.200C}
	\sum_{n \leq x} \mu(n) f(T^n(z)=o(x), 
\end{equation}	
where $T:Z\longrightarrow Z$ is a Borel measurable transformation, see \cite{SP2010}. Each of these conjectures has an extensive literature. \\

The best estimates for \eqref{eq5757.200A} in the literature have the asymptotic formulae of the forms
\begin{equation} \label{eq5757.210}
	\sum_{n \leq x} \mu(n) \mu(n+t)=O \left (\frac{x}{\sqrt{ \log \log x }} \right ),
\end{equation} 
where $t \ne0$ is a fixed integer, see \cite[Corollary 5]{HH2022}. Here, the following result is proposed. 

\begin{theorem} \label{thm5757MN.121}  \hypertarget{thm5757MN.121}Let $x\geq1$ be a large prime number, and let $\mu:\mathbb{N} \longrightarrow \{-1,0,1\}$ be the Mobius function. If $t\ne0$ is an integer, then, 
	\begin{equation} \label{eqP1199.800}
		\sum_{n \leq x} \mu(n) \mu(n+t) =O\bigg(e^{-c\sqrt{\log x}}\bigg)\nonumber, 
	\end{equation}	\nonumber
	where $c>0$ is a constant.
\end{theorem}	

The proof of \hyperlink{thm5757MN.121}{Theorem} \ref{thm5757MN.121}, based on a new idea and standard results in analytic number theory, is given in \hyperlink{SP1199}{Section} \ref{SP1199}.

\section{Average Orders of Mobius Functions}\label{S2222P}
\begin{theorem} \label{thm2222.500} \hypertarget{thm2222.500} If $\mu: \N\longrightarrow \{-1,0,1\}$ is the Mobius function, then, for any large number $x>1$, the following statements are true.
	\begin{enumerate} [font=\normalfont, label=(\roman*)]
		\item $\displaystyle \sum_{n \leq x} \mu(n)=O \left (xe^{-c\sqrt{\log x}}\right )$, \tabto{8cm} unconditionally,
		\item $\displaystyle \sum_{n\leq x}\frac{\mu(n)}{n}=O\left( e^{-c\sqrt{\log x}} \right ), $ \tabto{8cm} unconditionally,
	\end{enumerate}where $c>0$ is an absolute constant.
\end{theorem}
\begin{proof}[\textbf{Proof}]  See, \cite[p.\ 182]{MV2007}, \cite[p.\ 347]{HW2008}, et alii.   
\end{proof}

There are many sharp bounds of the summatory function of the Mobius function, say, $O(xe^{-c(\log x)^{\delta}})$, and the conditional estimate $O(x^{1/2+\varepsilon})$ presupposes that the nontrivial zeros of the zeta function $ \zeta(\rho)=0$ in the critical strip $\{0<\Re e(s)<1 \}$ are of the form $\rho=1/2+it, t \in \mathbb{R}$. However, the simpler notation will be used whenever it is convenient.

\section{Integers in Arithmetic Progressions}\label{A2002}
An effective asymptotic formula for the number of integers in arithmetic progressions is derived in \hyperlink{lemA2002.400W}{Lemma} \ref{lemA2002.400W}. The derivation is based on a version of the basic large sieve inequality stated below.
\begin{thm}\label{thmA2002.200W} \hypertarget{thmA2002.200W} Let $x$ be a large number and let $Q
	\leq x$. If $\{a_n:n\geq1\}$ is a sequence of real number, then
	\begin{equation}\label{eqA2002.100W}
		\sum_{q\leq Q}	q\sum_{1\leq a\leq q}\bigg |\sum_{\substack{n \leq x\\ n\equiv a \bmod q}}a_n-\frac{1}{q}\sum_{n \leq x}a_n\bigg|^2\leq Q\left(10Q+2\pi x \right) \sum_{n \leq x}|a_n|^2\nonumber.
	\end{equation}
\end{thm}
\begin{proof}[\textbf{Proof}] The essential technical details are covered in \cite[Chapter 23]{DH2000}. This inequality is discussed in \cite{GP1967} and the literature in the theory of the large sieve. 
\end{proof}
\begin{lem} \label{lemA2002.400W} \hypertarget{lemA2002.400W} If $x \geq 1$ is a large number and $1\leq a< q \leq x$, then
	\begin{equation}\label{eqA2002.400W}
		\max_{1\leq a\leq q}\bigg |\sum_{\substack{n \leq x\\ n\equiv a \bmod q}}1-\frac{1}{q}\sum_{n \leq x}1\bigg|=O\left(\frac{x}{q}e^{-c\sqrt{\log x} }\right),
	\end{equation}
	where $ c>0$ is a constant. In particular,
	\begin{equation}\label{eqA2002.405W}
		\sum_{\substack{n \leq x\\ n\equiv a \bmod q}}1=\left[\frac{x}{q}\right]+O\left(\frac{x}{q}e^{-c\sqrt{\log x} }\right).
	\end{equation}
	
\end{lem}
\begin{proof}[\textbf{Proof}] Trivially, the basic finite sum satisfies the asymptotic \begin{equation}\label{eqA2002.410W}
		\sum_{n \leq x}1=[x]= x-\{x\},
	\end{equation}
	where $[x]=x-\{x\}$ is the largest integer function, and the number of integers in any equivalent class satisfies the asymptotic formula
	\begin{equation}\label{eqA2002.415W}
		\sum_{\substack{n \leq x\\ n\equiv a \bmod q}}1	=	\frac{x}{q}+E(x).
	\end{equation}
	Let $Q=x$ and let the sequence of real numbers be $a_n=1$ for  $n\geq1$. Now suppose that the error term is of the form 
	\begin{equation}\label{eqA2002.420W}
		E(x)=E_0(x)=O\left(x^{\alpha}\right),
	\end{equation}
	where $ \alpha\in(0,1]$ is a constant. Then, the large sieve inequality,  \hyperlink{thmA2002.200W}{Theorem} \ref{thmA2002.200W}, yields the lower bound
	\begin{eqnarray}\label{eqA2002.430W}
		\sum_{q\leq x}	q\sum_{1\leq a\leq q}\bigg |\sum_{\substack{n \leq x\\ n\equiv a \bmod q}}1-\frac{1}{q}\sum_{n \leq x}1\bigg|^2
		&=&\sum_{q\leq x}	q\sum_{1\leq a\leq q}\bigg |\frac{x}{q}+O\left(x^{\alpha}\right)-\frac{x-\{x\}}{q}\bigg|^2\nonumber\\[.2cm]
		&\gg&\sum_{q\leq x}	q\sum_{1\leq a\leq q}\bigg |x^{\alpha}+\frac{\{x\}}{q}\bigg|^2\nonumber\\[.2cm]
		&\gg&\sum_{q\leq x}	q\sum_{1\leq a\leq q}\left |x^{\alpha}\right|^2\nonumber\\[.2cm]
		&\gg&x^{2\alpha}\sum_{q\leq x}q\sum_{1\leq a\leq q}1\nonumber\\[.2cm]
		&\gg&x^{2\alpha}\sum_{q\leq x}q^2\nonumber\\[.2cm]
		&\gg&	x^{3+2\alpha} .
	\end{eqnarray}
	On the other direction, it yields the upper bound
	\begin{eqnarray}\label{eqA2002.440W}
		\sum_{q\leq x}	q\sum_{1\leq a\leq q}\bigg |\sum_{\substack{n \leq x\\ n\equiv a \bmod q}}1-\frac{1}{q}\sum_{n \leq x}1\bigg|^2
		&\leq&	 Q\left(10Q+2\pi x \right) \sum_{n \leq x}|a_n|^2\\
		&\leq&	 x\left(10x+2\pi x \right) \sum_{n \leq x}|1|^2\nonumber\\
		&\ll&	 x^3\nonumber.
	\end{eqnarray}
	Clearly, the lower bound in \eqref{eqA2002.430W} contradicts the upper bound in \eqref{eqA2002.440W}. Similarly, the other possibilities for the error term
	\begin{equation}\label{eqA2002.445W}
		E_1=O\left(\frac{x}{(\log x)^c} \right)\quad \text{ and }\quad
		E_2=O\left(xe^{-c\sqrt{\log x} }\right),
	\end{equation}
	contradict large sieve inequality. Therefore, the error term is of the form
	\begin{equation}\label{eqA2002.450W}
		E(x)=O\left(\frac{x}{q}e^{-c\sqrt{\log x} }\right)=O\left(\frac{x}{q(\log x)^c} \right)=O\left(\frac{x}{q }\right),
	\end{equation}
	where $ c>0$ is a constant.
\end{proof}

\section{Proof of Theorem 1.1 }  \label{SP1199}\hypertarget{SP1199}
The proof explored in this section is based on a new technique from harmonic analysis. This technique uses the Ramanujan sum $c_q(n)$ and elementary analytic methods to derive an asymptotic formula for the Mobius autocorrelation function 
\begin{equation}\label{eqP1199.800a}
	R(t)=\sum_{n\leq x} \mu(n) \mu(n+t) . 
\end{equation}

\begin{proof}[\textbf{Proof}] ({\bfseries  \hyperlink{thm5757MN.121} {Theorem} \ref{thm5757MN.121}}) Without loss in generality, assume that $x\geq1$ is a large integer, and let $t=1$. Replace this identity, see \cite[Section 8.3]{AT1976},
		\begin{equation}\label{eqP1199.810a}
			c_n(1)	=\sum_{\substack{1\leq u<n\\\gcd(n,u)=1}}e^{i2 \pi u/n}
			=\mu(n)
		\end{equation}
		in \eqref{eqP1199.800a} twice, and substitute the characteristic function
		\begin{equation}\label{eqP1199.810b}
			\sum_{\substack{d\mid a\\d\mid n}}\mu(d)=
			\begin{cases}
				1&	\text{ if }\gcd(a,n)=1,\\
				0&	\text{ if }\gcd(a,n)\ne1,
			\end{cases}
		\end{equation}	
of relatively prime numbers. These substitutions transform \eqref{eqP1199.800a} into an exponential autocorrelation function:
		\begin{eqnarray}\label{eqP1199.810c}
			\sum_{n\leq x}\mu(n)\mu(n+1)&=& 	\sum_{n\leq x,}\sum_{\substack{1\leq u<n\\\gcd(n,u)=1}}e^{i2 \pi u/n}\times\sum_{\substack{1\leq v<n+1\\\gcd(n+1,v)=1}}e^{i2 \pi v/n+1}\\[.3cm]
			&=& 	\sum_{n\leq x,}\sum_{1\leq u<n}e^{i2 \pi u/n}\sum_{\substack{d_1\mid u\\d_1\mid n}}\mu(d_1)\times\sum_{1\leq v<n+1}e^{i2 \pi v/n+1}\sum_{\substack{d_2\mid v\\d_2\mid n+1}}\mu(d_2)\nonumber.
		\end{eqnarray}
Next, switch the order of summations
		\begin{align}\label{eqP1199.810d}
R(1)&=			\sum_{n\leq x}\mu(n)\mu(n+1)\\[.3cm]
&=	\sum_{n\leq x,}\sum_{d_1\mid n}\mu(d_1)\sum_{\substack{1\leq u<n\\d_1\mid u}}e^{i2 \pi u/n}\times \sum_{d_2\mid n+1}\mu(d_2)\sum_{\substack{1\leq v<n+1\\d_2\mid v}}e^{i2 \pi v/n+1}\nonumber\\[.3cm]
			&=\sum_{\substack{1\leq d_1< x\\1\leq d_2< x+1}}\mu(d_1)\mu(d_2)\sum_{\substack{1\leq n\leq x\\d_1\mid n,\; d_2\mid n+1}}	\sum_{\substack{1\leq u<n\\d_1\mid u}}e^{i2 \pi u/n}\sum_{\substack{1\leq v<n+1\\d_2\mid v}}e^{i2 \pi v/n+1}	\nonumber.
		\end{align}
Proceed to substitute the change of variables
		\begin{align}\label{eqP1199.810e}
			u&=d_1r , &n&=d_1k;\\[.2cm]
			v&=d_2s ,&n+1&=d_2m;
		\end{align}
to simplify \eqref{eqP1199.810d}. Specifically, 
		\begin{eqnarray}\label{eqP1199.820a}
			\sum_{n\leq x}\mu(n)\mu(n+1)
			&=&  \sum_{\substack{1\leq d_1< x\\1\leq d_2< x+1}}\mu(d_1)\mu(d_2)\sum_{\substack{1\leq n\leq x\\d_1\mid n,\; d_2\mid n+1}}	\sum_{1\leq r<k}e^{i2 \pi r/k}\sum_{1\leq s<m}e^{i2 \pi s/m}	\nonumber\\[.3cm]
			&=& \sum_{\substack{1\leq d_1< x\\1\leq d_2< x+1}}\mu(d_1)\mu(d_2)\sum_{\substack{1\leq n\leq x\\d_1\mid n,\; d_2\mid n+1}}1.
		\end{eqnarray}
		
The last equality follows from
		\begin{equation}\label{eqP1199.820b}
			\sum_{1\leq r<k}e^{i2 \pi r/k}
			=\sum_{1\leq s<m}e^{i2 \pi s/m}
			=-1
		\end{equation}
for any integers $k,m\geq2$. \\
		
Now, the conditions $d_1\mid n$ and $d_2\mid n+1$ imply that $\lcm(d_1,d_2)=d_1d_2$. Rearrange the last finite sum in the equivalent form
		\begin{eqnarray}\label{eqP1199.820c}
R(1)&=&			\sum_{n\leq x}\mu(n)\mu(n+1)\\[.3cm]
			&=&  \sum_{\substack{1\leq d_1< x\\1\leq d_2< x+1}}\mu(d_1)\mu(d_2)\sum_{\substack{1\leq n\leq x\\d_1\mid n,\; d_2\mid n+1}}1	\nonumber\\[.3cm]
			&=&  \sum_{\substack{1\leq d_1< x\\1\leq d_2< x+1}}\mu(d_1)\mu(d_2)\left(\sum_{\substack{1\leq n\leq x\\d_1\mid n,\; d_2\mid n+1}}1-\frac{x}{d_1d_2} +\frac{x}{d_1d_2} \right) 	\nonumber\\[.3cm]
			&=&  x\sum_{\substack{1\leq d_1< x\\1\leq d_2< x+1\\\gcd(d_1, d_2)=1}}\frac{\mu(d_1)\mu(d_2)}{d_1d_2} +\sum_{\substack{1\leq d_1< x\\1\leq d_2< x+1\\\gcd(d_1, d_2)=1}}\mu(d_1)\mu(d_2)\left(\sum_{\substack{1\leq n\leq x\\d_1\mid n,\; d_2\mid n+1}}1-\frac{x}{d_1d_2}  \right)  	\nonumber\\[.3cm]
			&=&R_0(x)\;+\;R_1(x)	\nonumber.
		\end{eqnarray}
The first finite sum has the upper bound
		\begin{align}\label{eqP1199.820d}
			R_0(x)&= x\sum_{\substack{1\leq d_1< x\\1\leq d_2< x+1\\\gcd(d_1, d_2)=1}}\frac{\mu(d_1)\mu(d_2)}{d_1d_2} \\[.3cm]
			&=x\sum_{\substack{1\leq d_1< x\\1\leq d_2< x+1}}\frac{\mu(d_1d_2)}{d_1d_2}\nonumber\\[.3cm]
			&= O\bigg(xe^{-c_1\sqrt{\log x}}\bigg)\nonumber,	
		\end{align}
where $d_1< x$ and $d_2< x+1$ are relatively prime and independent variables, this follows from  \hyperlink{thm2222.500}{Theorem} \ref{thm2222.500}. The upper bound of the second sum is computed in \hyperlink{lemP1199.400}{Lemma} \ref{lemP1199.400}. Summing these estimates yields
		\begin{eqnarray}\label{eqP1199.820e}
			\sum_{n\leq x}\mu(n)\mu(n+1)
			&=&  R_0(x)\;+\;R_1(x)\\
			&=&	O\bigg(xe^{-c_1\sqrt{\log x}}\bigg)\nonumber
		\end{eqnarray}
as claimed.
\end{proof}
	
\begin{lem} \label{lemP1199.400}\hypertarget{lemP1199.400} If $x$ is a large number, then
		\begin{equation}\label{eqP1199.400a}
			R_1(x)=	\sum_{\substack{1\leq d_1< x\\1\leq d_2< x+1\\\gcd(d_1, d_2)=1}}\mu(d_1)\mu(d_2)\left(\sum_{\substack{1\leq n\leq x\\d_1\mid n,\; d_2\mid n+1}}1-\frac{x}{d_1d_2}  \right) 
			=  O\bigg(xe^{-c_2\sqrt{\log x}}\bigg)  \nonumber,	
		\end{equation}
		where $c_2>0$ is a constant.
	\end{lem}
	\begin{proof}[\textbf{Proof}] Let $q=d_1d_2\leq x$. Taking absolute value and plugging in the asymptotic number of integers in the arithmetic progression, see \hyperlink{lemA2002.400W}{Lemma} \ref{lemA2002.400W}, into the inner sum yield the second finite sum has the upper bound,

		\begin{eqnarray}\label{eqP1199.400b}
			|R_1(x)|&\leq &	\left |	\sum_{\substack{1\leq d_1< x\\1\leq d_2< x+1\\\gcd(d_1, d_2)=1}}\mu(d_1)\mu(d_2)\left(\sum_{\substack{1\leq n\leq x\\d_1\mid n,\; d_2\mid n+1}}1-\frac{x}{d_1d_2}  \right)\right|\\ [.3cm]
			&\leq&\sum_{\substack{1\leq d_1< x\\1\leq d_2< x+1}}\left|\sum_{\substack{1\leq n\leq x\\d_1\mid n,\; d_2\mid n+1}}1-\frac{x}{d_1d_2}  \right|	\nonumber\\[.3cm]
			&\ll&\sum_{\substack{1\leq d_1< x\\1\leq d_2< x+1}}\frac{ x}{d_1d_2}e^{-c\sqrt{\log x}}	\nonumber\\[.3cm]
			&\ll&xe^{-c\sqrt{\log x}}\sum_{\substack{1\leq d_1< x\\1\leq d_2< x+1}}\frac{1}{d_1d_2}	\nonumber\\[.3cm]
			&=&  O\bigg(xe^{-c_2\sqrt{\log x}}\bigg)  \nonumber,	
		\end{eqnarray}
		where $c, c_2>0$ are constants, and the asymptotic estimate
		\begin{equation}\label{eqP1199.400c}
			\sum_{1\leq n\leq x}\frac{1}	{n}=O\bigg(e^{\log \log x}\bigg), 
		\end{equation}
		is used on the last line of expression \eqref{eqP1199.400b} to simplify the upper bound.
	\end{proof}



\end{document}